\documentclass{amsart}

\usepackage[latin1]{inputenc}
\usepackage[T1]{fontenc}
\usepackage{amsfonts,amsmath,amssymb}
\usepackage{graphicx}
\usepackage[all]{xy}

\newtheorem{theorem}{Theorem}[section]
\newtheorem{lemma}[theorem]{Lemma}
\newtheorem{corollary}[theorem]{Corollary}
\newtheorem{proposition}[theorem]{Proposition}

\theoremstyle{definition}
\newtheorem{definition}[theorem]{Definition}
\newtheorem{example}[theorem]{Example}

\theoremstyle{remark}
\newtheorem{remark}[theorem]{Remark}

\numberwithin{equation}{section}

\begin{document}

\title{Contractions of low-dimensional nilpotent Jordan algebras}

\author{J. M. Ancochea Berm\'{u}dez}
\address{Depto. Geometr\'{i}a y Topolog\'{i}a and IMI, Facultad de CC. Matem\'{a}ticas, Universidad Complutense de Madrid, Plaza de Ciencias, 3 28040 Madrid}
\thanks{E-mail: ancochea@mat.ucm.es; jfresan@estumail.ucm.es. The first author was supported in part by MTM 2006-09152 and CC607-UCM/ESP-2922.}

\author{J. Fres\'{a}n}

%    Information for third author
\author{J. Margalef Bentabol}

%    General info

\date{}


\begin{abstract}
In this paper we classify the laws of three-dimensional and four-dimensional nilpotent Jordan algebras over the field of complex numbers. We describe the irreducible components of their algebraic varieties and extend contractions and deformations among their isomorphism classes. In particular, we prove that $\mathcal{J}^2$ and $\mathcal{J}^3$ are irreducible and that $\mathcal{J}^4$ is the union of the Zariski closures of the orbits of two rigid Jordan algebras.
\end{abstract}

\maketitle

\section{Preliminaries}

Jordan algebras were introduced in 1934 in order to find a new mathematical setting for quantum mechanics. Since this foundational period, a complete structure theory has been developed from several points of view \cite{Mc}. However, apart from the Jordan-von Neumann-Wigner theorem on formally real algebras \cite{JNW} and Zel'manov's results on prime and simple Jordan algebras \cite{Zel}, the classification problem in a fixed dimension has been hardly studied and few explicit examples of Jordan algebras are known. In this paper we tackle the algebraic and geometric description of the varieties of three-dimensional and four-dimensional nilpotent Jordan algebras over the field of complex numbers. The structure of the work is as follows: we first define the notion of characteristic sequence and study its behaviour under contractions and deformations. In the next sections, three-dimensional and four-dimensional nilpotent Jordan algebras are classified according to the possible values of the characteristic sequence and some other invariants. We obtain diagrams summarizing the isomorphism classes, the contractions and deformations among them, the rigid Jordan algebras and the irreducible components of the varieties.

\begin{definition}
A complex Jordan algebra $J=(V, \varphi)$ consists of a complex vector space $V$ equipped with a symmetric bilinear mapping $\varphi: V \times V \longrightarrow V$ such that:
\begin{equation}
\varphi(\varphi(x,x), \varphi(x,y))=\varphi(x,\varphi(\varphi(x,x),y)) \quad x,y \in V.
\end{equation}
\end{definition}

\smallskip
Let $J=(V, \varphi)$ be a Jordan algebra. For any integer $m \in \mathbb{N}$ we define a lower central series as the descending chain of subalgebras
\begin{equation*}
\mathcal{C}^1(J)=J \supset \mathcal{C}^2(J)=\varphi(J, J)
\supset \ldots \mathcal{C}^{m+1}(J)=\varphi(\mathcal{C}^m(J),J)
\end{equation*}

\begin{definition}
A Jordan algebra $J=(V, \varphi)$ is said to be nilpotent if there exists an integer $k \in \mathbb{N}$ such that $\mathcal{C}^k(J)=\{0\}.$ The minimum $k$ for which this condition holds is the nilindex of $\varphi.$
\end{definition}

We recall that nilpotent Jordan algebras have non-trivial center, that is:  \begin{equation*}
Z(\varphi)=\{x \in V: \quad \varphi(x,y)=0 \quad \forall y \in V \}\neq \{0\}
\end{equation*}

Let $\{e_1,\ldots,e_n\}$ be a basis for $\mathbb{C}^n.$ It is possible to identify Jordan algebras with its structure constants over this basis, that is, to consider them in bijection with the set of symmetric tensors $(a_{ij}^k)$ whose coordinates satisfy the system of homogeneous polynomial constraints
\begin{equation}
a_{ii}^h a_{kj}^l a_{lk}^r - a_{ii}^h a_{jk}^l a_{lh}^r -2a_{ij}^h a_{ik}^l a_{hl}^r + 2a_{il}^r a_{hj}^l a_{ik}^h = 0, \quad 1\leq h,i,j,k,l,r\leq n.
\end{equation}

Nilpotency conditions are also of polynomial type and thus the set of nilpotent Jordan algebras in dimension $n,$ which will be denoted by $\mathcal{J}^n,$ possess a structure of algebraic variety embedded in $\mathbb{C}^{n^3}.$

Let $J=(\mathbb{C}^n, \varphi)$ be a nilpotent Jordan algebra. For any $x \in  \mathbb{C}^n,$ there is a natural way to define an endomorphism $L_x$ analogue to the adjoint representation:
\begin{equation*}
L_x(y)=\varphi(x,y) \quad \forall y \in \mathbb{C}^n
\end{equation*}

\smallskip
It is straightforward to prove that $L_x$ is nilpotent. In \cite{AnGo}, the concept of cha-racteristic sequence of a Lie algebra was introduced. This notion is formally still valid for Jordan algebras. If $x$ is a non-zero
vector not belonging to $\mathcal{C}^2(J),$ then let
$s_\varphi(x)=(s_1(x), \ldots, s_k(x))$ be the ordered descending sequence of dimensions of the Jordan blocks of $L_x.$ If we denote by $s(\varphi)$ the least
upper bound for the lexicographic order of the set $\{s_\varphi(x): \
x \in J - \mathcal{C}^2(J)\},$ then $s(\varphi)$ is an invariant of the
isomorphism class of $J$ called characteristic sequence. A
vector $x \in J - \mathcal{C}^2(J)$ is said to be a characteristic vector if
$s_\varphi(x)=s(\varphi).$

\subsection{Rigid Jordan algebras}
If $\varphi_0$ is a nilpotent Jordan algebra, let $\mathcal{O}(\varphi_0)$ be its orbit under the action of the general linear group $GL(n,\mathbb{C}):$
\begin{equation*}
\begin{array}{rcl}
GL(n,\mathbb{C}) \times \mathcal{J}^n& \longrightarrow &\mathcal{J}^n \\ (f, \varphi_0)
&\longmapsto & f^{-1}(\varphi_0(f(x),f(y)))
\end{array}
\end{equation*}

\smallskip
Let $C$ be an irreducible component of $\mathcal{J}^n$ which contains $\varphi_0.$ Then $\mathcal{O}(\mu_0) \subseteq C.$ If we endow the variety with the Zariski topology, then $C$ is closed, so the closure for this topology $\overline{\mathcal{O}(\mu_0)}^Z$ is also contained in $C.$

\begin{definition}
A Jordan algebra $\varphi$ is said to be rigid if $\mathcal{O}(\varphi)$ is Zariski open. In the same way, a family $\{\varphi_\alpha\}$ is rigid if the union of the orbits is Zariski open.
\end{definition}

Let $\varphi$ be a Jordan algebra and $f$ an automorphism, the image of the coboundaries
\begin{equation*}
\delta_\varphi f(x,y)=\varphi(f(x),y)+\varphi(x,f(y))-f(\varphi(x,y))
\end{equation*}
defines a structure of tangent space to the orbit of the algebra:
\begin{equation*}
T_\varphi\mathcal{O}(\varphi)=\{\delta_\varphi f \quad f \in GL(n,\mathbb{C})\}
\end{equation*}
This way, it is possible to assign $\mathcal{O}(\varphi)$ the vectorial dimension of its tangent space.

\begin{example}
Let $\{e_1,e_2,e_3\}$ be a basis for $\mathbb{C}^3.$ We compute $\dim \mathcal{O}(\varphi)$ for the Jordan algebra defined by the products $\varphi(e_1,e_1)=e_2,$ $\varphi(e_1,e_2)=e_3.$ Taking
\begin{equation*}
f=\left( \begin{matrix}
a_1 & b_1 & c_1 \\ a_2 & b_2 & c_2 \\ a_3 & b_3 & c_3 \\ \end{matrix} \right) \in GL(3,\mathbb{C})
\end{equation*} a generic automorphism we obtain the coboundaries:
\begin{align*}
&\delta_\varphi f(e_1,e_1)=-b_1e_1+(2a_1-b_2)e_2+(2a_2-b_3)e_3 \\
&\delta_\varphi f(e_1,e_2)=-c_1e_1+(b_1-c_2)e_2+(a_1+b_2-c_3)e_3 \\
\delta_\varphi f(e_1,e_3)&=c_1e_2+c_2e_3 \quad
\delta_\varphi f(e_2,e_2)=2b_1e_3 \quad
\delta_\varphi f(e_2,e_3)=c_1e_3,
\end{align*} which depend on seven parameters. Thus, $\dim \mathcal{O}(\varphi)=7.$
\end{example}

\begin{proposition}\cite{AnRu}, \cite{Ver}.
Let $J=(\mathbb{C}^n, \varphi)$ be a nilpotent Jordan algebra and $\mathrm{Der}(\varphi)$ its algebra of derivations. Then:
\begin{equation*}
\dim \mathcal{O}(\varphi)=n^2-\dim \mathrm{Der}(\varphi)
\end{equation*}
\end{proposition}

The geometric meaning of rigidity may be described by introducing a notion of limit in the variety $\mathcal{J}^n.$ Let $\varphi_0$ be a nilpotent Jordan
algebra and let $f_t \in GL(n,\mathbb{C})$ be a family of automorphisms depending on a continuous parameter $t.$ If the limit
\begin{equation}
\varphi(x,y)=\lim_{t\rightarrow 0}
f_{t}^{-1}(\varphi_0(f_{t}(x),f_{t}(y))) \label{lim}
\end{equation} exists for all $x, y \in \mathbb{C}^{n},$ then $\varphi$ is the contraction of $\varphi_0$ by $\{f_t\}.$ It is worthwhile to mention that properties defined by polynomial identities, such as nilpotency, associativity or commutativity, are all preserved by contraction.

\begin{example} \label{ex1}
Let $\{e_1,e_2,e_3\}$ be a basis for $\mathbb{C}^3,$ and $\varphi_0$ the Jordan algebra defined by $\varphi_0(e_1,e_1)=e_2$ and $\varphi_0(e_1,e_2)=e_3.$ If we consider the automorphism
\begin{equation*}
f_t(e_1)=te_1 \ \quad f_t(e_2)=t^2e_2 \quad f_t(e_3)=e_3,
\end{equation*} the only non-zero products in the transformed basis $x_i=f_t(e_i), \ i=1,2,3,$ are
\begin{equation*}
\varphi_0(x_1,x_1)=x_2, \quad \varphi_0(x_1,x_2)=t^3x_3.
\end{equation*} It is immediate that (\ref{lim}) holds for $e_1,e_2,e_3.$ In the limit, an algebra isomorphic to $\varphi(x_1,x_1)=x_2$ is obtained, so $\varphi$ is a contraction of $\varphi_0.$
\end{example}

Using the action of the general linear group over the variety, it is easy to prove that a contraction of $\varphi_0$ corresponds to a closure point of the orbit $\mathcal{O}(\varphi_0),$ so every irreducible component containing $\varphi_0$ contains also all the contractions of $\varphi_0.$ The change of basis $x_i=te_i,$ $i=1\ldots n,$ induces a contraction of every nilpotent Jordan algebra over the Abelian algebra. Moreover, for every nontrivial contraction $\varphi_0\longrightarrow \varphi,$ the following inequalities hold
\begin{equation} \label{ineq}
s(\varphi) \leq s(\varphi_0), \quad \dim \mathcal{O}(\varphi) < \dim \mathcal{O}(\varphi_0), \quad \dim Z(\varphi) \geq \dim Z(\varphi_0).
\end{equation}
It follows that rigid algebras cannot be obtained as a contraction of other non-isomorphic laws in $\mathcal{J}^n$ \cite{Ni}. Indeed, the Zariski closure of the orbit of a rigid Jordan algebra is an irreducible component of the variety.

\begin{proposition} \label{prop}
Let $\mathcal{J}^n$ be the variety of nilpotent Jordan algebras. The relation $\varphi_0 \longrightarrow \varphi$ is an order relation in $\mathcal{J}^n.$
\end{proposition}

\smallskip

Another possible approach to the geometric interpretation of rigidity, in some way dual to contractions, may be developed in the frame of deformation theory for algebraic structures.

\begin{definition}
Let $\varphi_0$ be a nilpotent Jordan algebra. A deformation of $\varphi$ is an uniparametric formal series
\begin{equation*}
\varphi_t=\varphi_0+\sum_{i=1}^\infty t^i \varphi_i
\end{equation*} where $\varphi_i: V \times V \longrightarrow V$ are symmetric bilinear mappings verifying the equations of $\mathcal{J}^n,$ that is, Jordan identity and nilpotency conditions.
\end{definition}

Given two deformations $\varphi_t^{1}$ and $\varphi_t^2$ of the Jordan algebra $\varphi_0,$ an equivalence relation between $\varphi_t^1$ and $\varphi_t^2$ may be defined by the existence of a linear isomorphism
\begin{equation*}
F_t=Id+\sum_{i=1}^\infty t^i g_i, \quad g_i \in GL(n,\mathbb{C})
\end{equation*} such that $\varphi_t^2(x,y)=F_t^{-1}(\varphi_t^1(F_t(x),F_t(y)))$ for all $x,y \in V.$

\smallskip

A deformation $\varphi_t$ of $\varphi_0$ is said to be trivial if it is equivalent to $\varphi_0.$ If $\varphi_0$ is deformed non-trivially over $\varphi,$ then
\begin{equation} \label{ineq2}
s(\varphi) \geq s(\varphi_0), \quad \dim \mathcal{O}(\varphi) > \dim \mathcal{O}(\varphi_0), \quad \dim Z(\varphi) \leq \dim Z(\varphi_0).
\end{equation} It follows that $\varphi_0$ is rigid if and only if any deformation $\varphi_t$ is trivial \cite{Ge}.

\section{The variety $\mathcal{J}^3$}

In dimension two, the only non-Abelian nilpotent Jordan algebra is given by the product $\varphi(e_1,e_1)=e_2.$ Thus, the variety $\mathcal{J}^2$ is irreducible \cite{AnRu}. In this paragraph we determine the isomorphism classes of $\mathcal{J}^3$, describe the irreducible components and extend contractions among their isomorphism classes.

\begin{theorem} [Classification of three-dimensional nilpotent Jordan algebras] Let $\varphi$ be a three-dimensional nilpotent complex Jordan
algebra. If $\varphi$ is not Abelian, then $\varphi$ is isomorphic to one of the following pairwise non-isomorphic algebras:

\begin{table}[ht]
\caption{}\label{eqtable}
\renewcommand\arraystretch{1.5}
\noindent\[
\begin{array}{ccccc}
\hline
\text{Isomorphism class} & & s(\varphi) & \dim \mathcal{O}(\varphi) & \dim Z(\varphi) \\
\hline
\varphi_1(e_1,e_1)=e_2 & \varphi_1(e_1,e_2)=e_3 & (3) & 7 & 1 \\
\hline
\varphi_2(e_1,e_1)=e_2 & \varphi_2(e_3,e_3)=e_2 & (2,1) & 6 & 1 \\
\hline
\varphi_3(e_1,e_1)=e_2 & & (2,1) & 4  & 2 \\
\hline
\end{array}
\]
\end{table}
\end{theorem}

\begin{proof}
Let $J=(\mathbb{C}^3, \varphi)$ be a nilpotent Jordan algebra.
Since $\varphi$ is non-Abelian, the possible values of the characteristic sequence are $(3)$ and $(2,1).$ If $s(\varphi)=(3),$ there exists a characteristic vector $e_1$ and a basis $\{e_1,e_2,e_3\}$ for $\mathbb{C}^3$ such that
\begin{equation*}
\varphi(e_1,e_1)=e_2, \quad \varphi(e_1,e_2)=e_3.
\end{equation*} When applied to $x=y=e_1,$ Jordan
identity shows that $\varphi(e_2,e_2)=0$ and, being $Z(\varphi)$ non-trivial, the operator $L_{e_3}$ must be identically null. Thus, there is only one nilpotent Jordan algebra defined by
\begin{equation*}
\varphi_1(e_1,e_1)=e_2, \quad \varphi_1(e_1,e_2)=e_3.
\end{equation*}

If $s(\varphi)=(2,1),$ there exists a characteristic vector $e_1$
such that $\varphi(e_1,e_1)$ is non-zero. By contradiction, assume
that $\varphi(x, x)=0$ holds for every characteristic vector.
Then, we can find a basis $\{e_1,e_2,e_3\}$ such that the only
non-zero product is  $\varphi(e_1,e_2)=e_3.$ However, if we
consider
\begin{equation} \label{changebasis}
x_1=\frac{1}{\sqrt{2}}(e_1+e_2), \quad x_2=e_3, \quad x_3=-\frac{1}{\sqrt{2}}i(e_1-e_2),
\end{equation} then $s_\varphi(x_1)=s(\varphi)$ and $\varphi(x_1,x_1)=x_2.$ Thus, there exists a basis $\{e_1,e_2,e_3\}$ for $\mathbb{C}^3$ such that $\varphi(e_1,e_1)=e_2.$ Jordan identity evaluated in $x=y=e_1$ shows that $\varphi(e_2,e_2)=0$ and nilpotency implies $\varphi(e_2,e_3)=0$ and $\varphi(e_3,e_3)=ae_2$ for some $a \in \mathbb{C}.$ We claim that the nullity or not of this parameter defines the isomorphism class, since the dimension of the center depends on it. If $a \neq 0,$ the change of basis $x_1=e_1,$ $x_2=e_2$ and $x_3=\frac{1}{\sqrt{a}}e_3$ yields to
\begin{equation*}
\varphi_2(e_1,e_1)=e_2, \quad \varphi_2(e_3,e_3)=e_2.
\end{equation*} Otherwise, $\varphi$ is isomorphic to the Jordan algebra defined by
\begin{equation*}
\varphi_3(e_1,e_1)=e_2. \qedhere
\end{equation*} \end{proof}

The classification given above may be easily adapted to the variety of three-dimensional nilpotent Jordan algebras over the field of real numbers. We note that the condition $\varphi(x,x)=0$ for every characteristic vector holds for $\varphi_4(e_1,e_2)=e_3,$ since the change of basis (\ref{changebasis}) is not possible anymore. Indeed, when the parameter $a \in \mathbb{R}$ of the proof is negative, another isomorphism class is obtained:
\begin{align*}
\varphi_5(e_1,e_1)&=e_2, \quad \varphi_5(e_3,e_3)=-e_2,
\end{align*} These algebras are non-isomorphic to any of the former ones. In particular, $\varphi_4$ is not associative and $\dim \mathcal{O}(\varphi_4)=5$, $\dim \mathcal{O}(\varphi_5)=6.$

\smallskip

\subsection{Irreducible components}

\begin{theorem}
The variety $\mathcal{J}^3$ is irreducible.
\end{theorem}

\begin{proof}
Since $\varphi_1$ is the only Jordan algebra with maximal characteristic sequence, $\varphi_1$ is rigid and the Zariski closure of its orbit forms an irreducible component of the variety. It suffices to prove that $\varphi_2$ and $\varphi_3$ belong to $\overline{\mathcal{O}(\varphi_1)}^Z.$ It has already been shown (example \ref{ex1}) that $\varphi_3$ is a contraction of $\varphi_1.$ Moreover, if we consider the linear transformations
\begin{equation*}
f_t(e_1)=te_1, \quad f_t(e_2)=e_2, \quad f_t(e_3)=te_3.
\end{equation*} the algebra $\varphi(x_1,x_2)=x_3$ is obtained in the limit. Since $\varphi$ and $\varphi_2$ are isomorphic, $\varphi_2$ is a contraction of $\varphi_1.$
\end{proof}

\begin{remark}
The variety $\mathcal{J}^3$ is still irreducible when considered over the reals, because the families of automorphisms
\begin{align*}
f_t(e_1)&=te_1 \quad f_t(e_2)=te_2 \ \ \ \quad f_t(e_3)=t^2e_3 \\
g_t(e_1)&=te_3 \quad g_t(e_2)=-t^2e_2 \quad g_t(e_3)=te_1
\end{align*} induce contractions of $\varphi_1$ over the real Jordan algebras $\varphi_4$ and $\varphi_5$ respectively.
\end{remark}

In order to study the relation between $\mathcal{J}^3$ and $\mathcal{A}^3,$ the variety of three-dimensional nilpotent associative algebras, we state the structure theorem \cite{Ma} choosing convenient representatives of each isomorphism class:

\begin{theorem} Let $(\mathbb{C}^3, \beta)$ be a nilpotent associative algebra. If $\beta$ is non-Abelian, then $\beta$ is isomorphic to one of the following associative algebras.
\begin{enumerate}
\item $\ \beta_1(e_1,e_1)=e_2 \ \ \quad \beta_1(e_1,e_2)=e_3 \quad \ \beta_1(e_2,e_1)=e_3$
\item $\beta_2^{\mu}(e_1,e_1)=e_2 \ \ \quad \beta_2^{\mu}(e_1,e_3)=e_2 \quad \ \beta_2^{\mu}(e_3,e_3)=\mu e_2 \ \ \quad \mu \in \mathbb{C}$
\item $\ \beta_3(e_1,e_1)=e_2 \ \ \quad \beta_3(e_3,e_3)=e_2$
\item $\ \beta_4(e_1,e_2)=-e_3 \quad \beta_4(e_2,e_1)=e_3$
\item $\ \beta_5(e_1,e_1)=e_2$
\end{enumerate} Indeed, the algebra $\beta_1$ and the family $\{\beta_2^{\mu}\}_{\mu\neq 0}$ are rigid.
\end{theorem} It follows from the theorem that $\mathcal{J}^3 \subset \mathcal{A}^3.$ The converse is false, but if we consider the classical correspondence between associative and Jordan algebras
\begin{equation} \label{corr}
\varphi(x,y)=\frac{1}{2}[\beta(x,y)+\beta(y,x)]
\end{equation} the image of $\mathcal{A}^3$ coincides with $\mathcal{J}^3.$ In effect, $\beta_4$ is a Lie algebra, so it yields to $\varphi \equiv 0$ and if $\beta$ belongs to $\{\beta_2^{\mu}\},$ then
\begin{equation*}
\varphi(e_1,e_1)=e_2 \quad \varphi(e_1,e_3)=\frac{1}{2}e_2 \quad \varphi(e_3,e_3)=\mu e_2,
\end{equation*} which is isomorphic to $\varphi_2.$ Here we have an example of a rigid family of associative algebras whose Jordan product is not rigid, so (\ref{corr}) does not preserve rigidity.

\section{The variety $\mathcal{J}^4$}

In this paragraph, we determine the structure of four-dimensional nilpotent Jordan algebras in function of the possible values of the characteristic sequence.

\subsection{Characteristic sequence $s(\varphi)=(4)$}

If $s(\varphi)=(4),$ there exists a basis $\{e_1,e_2,e_3,e_4\}$ for $\mathbb{C}^4$ such that the center of the algebra is spanned by $\{e_4\}$ and the following relations hold:
\begin{equation*}
\varphi(e_1,e_1)=e_2, \quad \varphi(e_1,e_2)=e_3, \quad \varphi(e_1,e_3)=e_4.
\end{equation*} It follows from Jordan identity that $\varphi(e_2,e_2)=e_4$ and $\varphi(e_2,e_3)=0.$ Assume $\varphi(e_3,e_3)=ae_2+be_3+ce_4.$ Nilpotency implies the nullity of $b$ and Jordan identity applied to $x=e_3, y=e_1$ yields to $a=0.$ In fact, $c$ must be zero because, otherwise, if we set $x=e_1+e_3,$ the constraints for the structure constants of $\varphi(x,x)$ would not be satisfied. Thus, $\varphi$ is isomorphic to the Jordan algebra defined by the law
\begin{equation}
\varphi_1(e_i,e_j)=e_{i+j} \quad 2 \leq i+j \leq 4.
\end{equation}
This way, there is only an isomoprhism class when the characteristic sequence is maximal. In arbitrary dimension, this algebra comes to be a rigid intersection point of the varieties of associative and Jordan algebras.

\subsection{Characteristic sequence $s(\varphi)=(3,1)$}

\begin{lemma} \label{char.seq.3.1}
Let $(\mathbb{C}^4, \varphi)$ be a four-dimensional complex nilpotent Jordan algebra. If $s(\varphi)=(3,1),$ there exists a basis $\{e_1,e_2,e_3,e_4\}$ such that
\begin{align*}
\varphi(e_1,e_1)=e_2, \quad \varphi(e_1,e_2)=e_3, \quad \varphi(e_1,e_3)=\varphi(e_1,e_4)=0, \\
\varphi(e_2,e_2)=\varphi(e_2,e_3)=0, \quad \varphi(e_2,e_4)=ae_3, \quad \varphi(e_3,e_3)=be_3+ce_4, \\ \varphi(e_3,e_4)=de_2+fe_3+ge_4, \quad \quad
\varphi(e_4,e_4)=he_2+ie_3+je_4,
\end{align*} for some complex parameters $a,b,c,d,e,f,g,h,i,j.$
\end{lemma}
\begin{proof}
We first prove that when the characteristic sequence takes the value $(3,1),$ there exist a basis $\{e_1,e_2,e_3,e_4\}$ for $\mathbb{C}^4$ such that
\begin{equation} \label{basis}
\varphi(e_1,e_1)=e_2, \quad \varphi(e_1,e_2)=e_3, \quad \varphi(e_1,e_3)=\varphi(e_1,e_4)=0.
\end{equation}
In effect, in this case there is a characteristic vector $x_1$ such that $\dim Im R_{x_1}=2$ and $\dim (Im R_{x_1})^2=1.$ If $\varphi(x_1,x_1) \in \mathcal{C}^2(J) - \mathcal{C}^3(J),$ then we take
\begin{equation*}
e_1=x_1, \quad e_2=\varphi(e_1,e_1), \quad e_3=\varphi(e_1,e_2),
\end{equation*}
and complete the basis with some $e_4 \in \ker (Im R_{x_1})$ linearly independent with $e_3.$ Otherwise, $\varphi(x_1,x_1) \in \mathcal{C}^3(J)$ and there exists $x_2 \notin \mathcal{C}^2(J)$ such that $\varphi(x_1,x_2)=x_3$ and $\varphi(x_1,x_3)=x_4.$ Thus $\mathcal{C}^2(J)$ is spanned by $\{x_3,x_4\}$ and $\mathcal{C}^3(J)$ by $\{x_4\}.$ Then $\varphi(x_1,x_1)=ax_4$ for some $a \in \mathbb{C}$ and we may choose $\alpha \in \mathbb{C}$ such that the vectors
\begin{equation*}
e_1=x_1+\alpha x_2 \quad e_2=\varphi(e_1,e_1) \quad e_3=\varphi(e_1,e_2)
\end{equation*} are linearly independent. To complete the basis, just take $e_4 \in \ker (Im R_{x_1})$ linearly independent with $e_3.$ Proven the existence of such a basis, the rest of the lemma follows from nilpotency and Jordan identity.
\end{proof}

Since $e_1$ and $e_2$ are not in the center of the algebra, $Z(\varphi)$ is one or two-dimensional. It is easy to check that if $\dim Z(\varphi)=1$ then the center is spanned by $\{e_3\}.$ By contradiction, assume $Z(\varphi)$ spanned by $\{e_4\}.$ It follows that all the parameters of the lemma are zero, except for $c.$ But then, considering the basis $\{x_1=e_1+e_2, x_2=e_2+ce_4, x_3=e_3, x_4=ce_4\}$ we have
\begin{equation*}
\varphi(x_1,x_i)=x_{i+1} \quad i=1,2,3
\end{equation*} so $s_{\varphi}(x_1)=(4)$ and $s(\varphi)$ would be maximal. Thus, $Z(\varphi)$ is spanned by $\{e_3\}$ or by $\{e_3,e_4\}.$ In the first case, the only a priori non-zero products are
\begin{align*}
\varphi(e_1,e_1)=e_2 \quad \quad \varphi(e_1,e_2)=e_3 \quad \varphi(e_2,e_4)=\gamma e_3 \quad \varphi(e_4,e_4)=\alpha e_2+\beta e_3
\end{align*} where the parameter $j \in \mathbb{C}$ has been eliminated because of nilpotency. Considering the quotient $\varphi / Z(\varphi)$ it is straightforward to prove that the nullity of $\alpha$ is invariant under isomorphism. First assume $\alpha \neq 0.$
\smallskip
\begin{itemize}
\item If $\alpha+\gamma^2 \neq 0,$ then two changes of basis yield to
\begin{equation*}
\varphi_2(e_1,e_1)=e_2, \quad \varphi_2(e_1,e_2)=e_3, \quad \varphi_2(e_4,e_4)=e_2.
\end{equation*}
\item Otherwise, we can reduce the Jordan algebra to the form
\begin{align*}
\varphi(x_1,&x_1)=x_2, \quad \varphi(x_1,x_2)=x_3, \quad \varphi(x_2,x_4)=x_3, \\ &\varphi(x_4,x_4)=-x_2+\beta\frac{1-c}{\gamma^3 (1+c)^2}x_3,
\end{align*} for some $c \in \mathbb{C}.$ If $\beta \neq 0,$ with a convenient choice of the parameter
\begin{align*}
\varphi_3(e_1,e_1)=e_2, \quad \varphi_3(e_1,e_2)=e_3, \quad \varphi_3(&e_2,e_4)=e_3, \quad \varphi(e_4,e_4)=-e_2-e_3.
\end{align*} However, if $\beta$ is zero, then $\varphi$ is isomorphic to the law \begin{align*}
\varphi_4(e_1,e_1)=e_2, \quad \varphi_4(e_1,e_2)=e_3, \quad \varphi_4(e_2,e_4)&=e_3, \quad \varphi_4(e_4,e_4)=-e_2.
\end{align*}
\end{itemize}

Now assume $\alpha=0.$
\begin{itemize}
\item If $\gamma \neq 0,$ then two changes of basis yield to
\begin{equation*}
\varphi_5(e_1,e_1)=e_2, \quad \varphi_5(e_1,e_2)=e_3, \quad \varphi_5(e_2,e_4)=e_3.
\end{equation*}

\item If $\gamma =0,$ then $\beta$ must be non-zero because, otherwise, $e_4 \in Z(\varphi).$ In this case, the change of basis $x_4=\frac{1}{\sqrt{\beta}}e_4$ induces an isomorphism with
\begin{equation*}
\varphi_6(e_1,e_1)=e_2, \quad \varphi_6(e_1,e_2)=e_3, \quad \varphi_6(e_4,e_4)=e_3.
\end{equation*}
\end{itemize}

\smallskip

In the second case, that is, when $\dim Z(\varphi)=2,$ there are no free parameters in the lemma, so $\varphi$ is isomorphic to the Jordan algebra defined by the law
\begin{equation*}
\varphi_7(e_1,e_1)=e_2, \quad \varphi_7(e_1,e_2)=e_3.
\end{equation*}

\subsection{Characteristic sequence $s(\varphi)=(2,2)$}
\begin{lemma} \label{char.seq.2.2}
Let $(\mathbb{C}^4, \varphi)$ be a four-dimensional complex nilpotent Jordan algebra. If $s(\varphi)=(2,2),$ then $\dim Z(\varphi)=2.$
\end{lemma}

\begin{proof}
If $s(\varphi)=(2,2),$ there exists a basis $\{e_1,e_2,e_3,e_4\}$ such that
\begin{equation*}
\varphi(e_1,e_1)=e_2, \quad \varphi(e_1,e_3)=e_4, \quad \varphi(e_1,e_2)=\varphi(e_1,e_4)=0
\end{equation*} Thus, $Z(\varphi)$ is one or two-dimensional. Assume that the center is spanned by $\{e_4\}.$ Nilpotency and Jordan identity show that
\begin{equation*}
\varphi(e_2,e_3)=ae_4, \quad \varphi(e_3,e_3)=be_2+ce_4,
\end{equation*} where $a$ is necessarily non-zero. But then, $s_\varphi(x)=(3,1)$ holds for $x=\sqrt{1-b}e_1+e_3.$ An analogous reasoning proves that $Z(\varphi)$ cannot be spanned by $\{e_2\}.$
\end{proof}

According to the lemma, $e_2,e_4 \in Z(\varphi),$ so the only unknown parameters are $b$ and $c.$ It is not difficult to prove that the nullity of the quantity $\delta=b+\frac{c^2}{4}$ is invariant under isomorphism. If $\delta$ is non-zero, a suitable change of basis yield to
\begin{equation*}
\varphi_8(e_1,e_1)=e_2, \quad \varphi_8(e_3,e_3)=e_4.
\end{equation*} Otherwise we obtain the Jordan algebra
\begin{equation*}
\varphi_9(e_1,e_1)=e_2, \quad \varphi_9(e_1,e_3)=e_4.
\end{equation*}

\subsection{Characteristic sequence $s(\varphi)=(2,1,1)$}
\begin{lemma}
Let $(\mathbb{C}^4, \varphi)$ be a four-dimensional complex nilpotent Jordan algebra with characteristic sequence $s(\varphi)=(2,1,1).$ Then the isomorphism class of $\varphi$ is determined by the dimension of the center. More precisely:
\begin{align*}
\dim Z(\varphi)=1 \quad \quad &\varphi_{10}(e_1,e_1)=e_2, \quad \varphi_{10}(e_3,e_4)=e_2. \\
\dim Z(\varphi)=2 \quad \quad &\varphi_{11}(e_1,e_1)=e_2, \quad \varphi_{11}(e_3,e_3)=e_2. \\
\dim Z(\varphi)=3 \quad \quad &\varphi_{12}(e_1,e_1)=e_2.
\end{align*}
\end{lemma}
\begin{proof}
If $s(\varphi)=(2,1,1)$ there exists a basis $\{e_1,e_2,e_3,e_4\}$ for $\mathbb{C}^4$ such that
\begin{equation*}
\varphi(e_1,e_1)=e_2, \quad \varphi(e_1,e_i)=0, \quad i=2,3,4.
\end{equation*}

If the center is one-dimensional, an analogous argument to the one exposed in the proof of lemma \ref{char.seq.2.2} shows that $Z(\varphi)$ must be spanned by $\{e_2\}.$ Nilpotency and Jordan identity yield to
\begin{equation*}
\varphi(e_3,e_3)=ae_2, \quad \varphi(e_3,e_4)=be_3, \quad \varphi(e_4,e_4)=ce_2.
\end{equation*} Independently from the nullity of the parameters, there is a change of basis which induces an isomorphism with
\begin{equation*}
\varphi_{10}(e_1,e_1)=e_2, \quad \varphi_{10}(e_3,e_4)=e_2.
\end{equation*}

If $\dim Z(\varphi)=2,$ the center is spanned by $\{e_2,e_4\}$ and the only product to be determined is $\varphi(e_3,e_3)=ae_2+be_4.$ In fact, $b=0$, because otherwise we choose $\alpha \in \mathbb{C}$ such that $s_\varphi(\alpha e_1+e_3)=(2,2).$ Thus, $\varphi$ is isomorphic to
\begin{equation*}
\varphi_{11}(e_1,e_1)=e_2, \quad \varphi_{11}(e_3,e_3)=e_2.
\end{equation*}

When $\dim Z(\varphi)=3,$ it is obvious that $\varphi$ is isomorphic to $\varphi_{12}(e_1,e_1)=e_2.$ \qedhere
\end{proof}

\begin{theorem} [Classification of four-dimensional nilpotent Jordan algebras] Let $\varphi$ be a four-dimensional nilpotent complex Jordan
algebra. If $\varphi$ is not Abelian, then $\varphi$ is isomorphic to one of the following pairwise non-isomorphic algebras:

\begin{table}[ht]
\caption{}\label{eqtable}
\renewcommand\arraystretch{1}
\noindent\[
\begin{array}{cccccc}
\hline
\text{Isomorphism class} & & & s(\varphi) & \mathcal{O}(\varphi) & Z(\varphi) \\
\hline
\varphi_1(e_1,e_1)=e_2 & \varphi_1(e_1,e_2)=e_3 & \varphi_1(e_1,e_3)=e_4 & & & \\ \varphi_1(e_2,e_2)=e_4 & & & (4) & 12 & 1 \\
\hline
\varphi_2(e_1,e_1)=e_2 & \varphi_2(e_1,e_2)=e_3 & \varphi_2(e_4,e_4)=e_2 & (3,1) & 13 & 1\\
\hline
\varphi_3(e_1,e_1)=e_2 & \varphi_3(e_1,e_2)=e_3 & \varphi_3(e_2,e_4)=e_3 & & & \\
\varphi_3(e_4,e_4)=-e_2-e_3 & & & (3,1) & 12 & 1\\
\hline
\varphi_4(e_1,e_1)=e_2 & \varphi_4(e_1,e_2)=e_3 & \varphi_4(e_2,e_4)=e_3 & & & \\
\varphi_4(e_4,e_4)=-e_2 & & & (3,1) & 11 & 1\\
\hline
\varphi_5(e_1,e_1)=e_2 & \varphi_5(e_1,e_2)=e_3 & \varphi_5(e_2,e_4)=e_3 & (3,1) & 12 & 1 \\
\hline
\varphi_6(e_1,e_1)=e_2 & \varphi_6(e_1,e_2)=e_3 & \varphi_6(e_4,e_4)=e_3 & (3,1) & 11 & 1\\
\hline
\varphi_7(e_1,e_1)=e_2 & \varphi_7(e_1,e_2)=e_3 & & (3,1) & 10 & 2 \\
\hline
\varphi_8(e_1,e_1)=e_2 & \varphi_8(e_3,e_3)=e_4 & & (2,2) & 10 & 2 \\
\hline
\varphi_9(e_1,e_1)=e_2 & \varphi_9(e_1,e_3)=e_4 & & (2,2) & 9 & 2 \\
\hline
\varphi_{10}(e_1,e_1)=e_2 & \varphi_{10}(e_3,e_4)=e_2 & & (2,1,1) & 9 & 1 \\
\hline
\varphi_{11}(e_1,e_1)=e_2 & \varphi_{11}(e_3,e_3)=e_2 & & (2,1,1) & 8 & 2 \\
\hline
\varphi_{12}(e_1,e_1)=e_2 & & & (2,1,1) & 6 & 3\\
\hline
\end{array}
\]
\end{table}
\end{theorem}
\begin{remark}
With the notation of the theorem, the associative algebras in the variety of $\mathcal{J}^4$ are $\varphi_1$ and $\varphi_i$ for $i\geq 6.$
\end{remark}

\subsection{Irreducible components}

\begin{theorem}
The variety $\mathcal{J}^4$ is the union of two irreducible components
\begin{equation*}
\mathcal{J}^4=\overline{\mathcal{O}(\varphi_1)}^Z \cup \overline{\mathcal{O}(\varphi_2)}^Z.
\end{equation*}
\end{theorem}

\begin{proof}
From the inequalities (\ref{ineq2}), it is straightforward to prove the rigidity of the Jordan algebras $\varphi_1$ and $\varphi_2.$ In effect, $\varphi_1$ is the only algebra with maximal characteristic sequence, so every deformation must be equivalent to $\varphi_1.$ In the same way, the rigidity of $\varphi_2$ follows from the fact that $\dim \mathcal{O}(\varphi_2)$ takes the biggest value of the table. Thus, $\overline{\mathcal{O}(\varphi_1)}^Z$ and $\overline{\mathcal{O}(\varphi_2)}^Z$ are irreducible components of $\mathcal{J}^4.$

We claim that $\varphi_7,$ $\varphi_9,$ $\varphi_{10},$ $\varphi_{11},$ $\varphi_{12}$ belong to the Zariski closure of $\mathcal{O}(\varphi_1).$ Since $\varphi_1$ is associative, so are all the closure points of $\mathcal{O}(\varphi_1).$ Thus, $\varphi_2,$ $\varphi_3,$ $\varphi_4$ and $\varphi_5$ are not contractions of $\varphi_1.$ The study for nilpotent associative algebras accomplished in reference \cite{Ma} shows that $\varphi_8$ cannot be deformed over algebras of larger nilindex, so $\varphi_8$ is not a contraction of $\varphi_1.$ Neither is $\varphi_6,$ because of the transitivity established by proposition \ref{prop}. In effect, $\varphi_6 \longrightarrow \varphi_8$ by the automorphisms
\begin{equation*}
f_t(e_1)=te_1 \quad f_t(e_2)=t^2e_2 \quad f_t(e_3)=e_4 \quad f_t(e_4)=e_3.
\end{equation*} so if $\varphi_6$ were a contraction of $\varphi_1,$ so would be $\varphi_8.$ The remaining contractions are given by the following families of linear transformations or by composition of them:
\begin{align*}
\varphi_1 \longrightarrow \varphi_7: \quad &f_t(e_1)=te_1, \quad f_t(e_2)=t^2e_2, \quad f_t(e_3)=t^3e_3, \quad f_t(e_4)=e_4. \\
\varphi_1 \longrightarrow \varphi_9: \quad &f_t(e_1)=te_1, \quad f_t(e_2)=t^2e_2, \quad f_t(e_3)=te_3, \quad f_t(e_4)=t^2e_4. \\
\varphi_1 \longrightarrow \varphi_{10}: \quad &f_t(e_1)=te_2, \quad f_t(e_2)=t^2e_4, \quad f_t(e_3)=te_1, \quad f_t(e_4)=te_3. \\
\varphi_9 \longrightarrow \varphi_{11}: \quad &f_t(e_1)=e_1+e_3, \quad f_t(e_2)=e_2+2e_4, \quad f_t(e_3)=te_1+te_3, \\ \quad &f_t(e_4)=t^2e_2+2t^2e_4. \\
\varphi_{10} \longrightarrow \varphi_{11}: \quad &f_t(e_1)=e_1, \quad f_t(e_2)=e_2, \quad f_t(e_3)=e_3, \quad f_t(e_4)=te_4.\\
\varphi_{11} \longrightarrow \varphi_{12}: \quad &f_t(e_1)=e_1, \quad f_t(e_2)=e_2, \quad f_t(e_3)=te_3, \quad f_t(e_4)=e_4. \\
\end{align*}
Now let us show that $\varphi_3,$ $\varphi_4,$ $\varphi_5,$ $\varphi_6,$ $\varphi_8$ are contractions of $\varphi_2.$ We first note that we have the linear deformations
\begin{equation*}
\varphi_3+t\mu_1, \quad \varphi_4+t\mu_2, \quad \varphi_6+t\mu_1, \quad \varphi_4+t\mu_2
\end{equation*} where $\mu_1$ and $\mu_2$ are symmetric bilinear mappings with non-zero products $\mu_1(e_2,e_4)=e_3$ and $\mu_2(e_4,e_4)=e_3$ respectively. This way, $\varphi_2$ is a deformation of $\varphi_3,$ $\varphi_4$ may be deformed over $\varphi_3$ and $\varphi_5$ is a deformation of both $\varphi_4$ and $\varphi_6.$ It follows that $\varphi_3, \varphi_4, \varphi_5, \varphi_6$ are closure points of $\mathcal{O}(\varphi_2).$

The remaining contractions in the variety $\mathcal{J}^4$ are given by the following families of automorphisms or by composition of them:
\begin{align*}
\varphi_2 \longrightarrow \varphi_5: \quad &f_t(e_1)=te_1+e_4, \quad f_t(e_2)=(t^2+1)e_2, \quad f_t(e_3)=t(t^2+1)e_3, \\
\quad &f_t(e_4)=te_1+te_4. \\
\varphi_3 \longrightarrow \varphi_6: \quad &f_t(e_1)=te_1, \quad f_t(e_2)=t^2e_2, \quad f_t(e_3)=t^3e_3, \quad f_t(e_4)=it^{\frac{3}{3}}e_4. \\
\varphi_4 \longrightarrow \varphi_{7}: \quad &f_t(e_1)=e_1, \quad f_t(e_2)=e_2, \quad f_t(e_3)=e_3, \quad f_t(e_4)=te_4. \\
\varphi_4 \longrightarrow \varphi_{8}: \quad &f_t(e_1)=te_1+te_2, \quad f_t(e_2)=t^2e_3, \quad f_t(e_3)=it^{\frac{3}{2}}e_4, \quad f_t(e_4)=t^3e_4. \\
\varphi_4 \longrightarrow \varphi_{10}: \quad &f_t(e_1)=te_1, \quad f_t(e_2)=t^2e_2, \quad f_t(e_3)=t^2e_1+e_3, \quad f_t(e_4)=ite_4. \\
\varphi_6 \longrightarrow \varphi_7: \quad &f_t(e_1)=e_1, \quad f_t(e_2)=e_2, \quad f_t(e_3)=e_3, \quad f_t(e_4)=te_4. \\
\varphi_6 \longrightarrow \varphi_8: \quad &f_t(e_1)=te_1, \quad f_t(e_2)=t^2e_2, \quad f_t(e_3)=e_4, \quad f_t(e_4)=e_3. \\
\varphi_6 \longrightarrow \varphi_{10}: \quad &f_t(e_1)=te_4, \quad f_t(e_2)=t^2e_3, \quad f_t(e_3)=t^2e_1, \quad f_t(e_4)=e_2. \\
\varphi_{7} \longrightarrow \varphi_{9}: \quad &f_t(e_1)=te_1, \quad f_t(e_2)=t^2e_2, \quad f_t(e_3)=t^2e_1+e_2+e_4, \quad f_t(e_4)=te_3.\\
\varphi_8 \longrightarrow \varphi_{9}: \quad &f_t(e_1)=e_1+te_3, \quad f_t(e_2)=e_2+t^2e_4, \quad f_t(e_3)=t^2e_3, \quad f_t(e_4)=t^3e_4.
\end{align*} This completes the geometric description of the variety.
\end{proof}

\begin{corollary}
All contractions in the variety of four-dimensional nilpotent
Jordan algebras are represented by the arrows in the following
diagram, where $\varphi_{13}$ stands for the Abelian algebra, or
by composition of them.
\end{corollary}

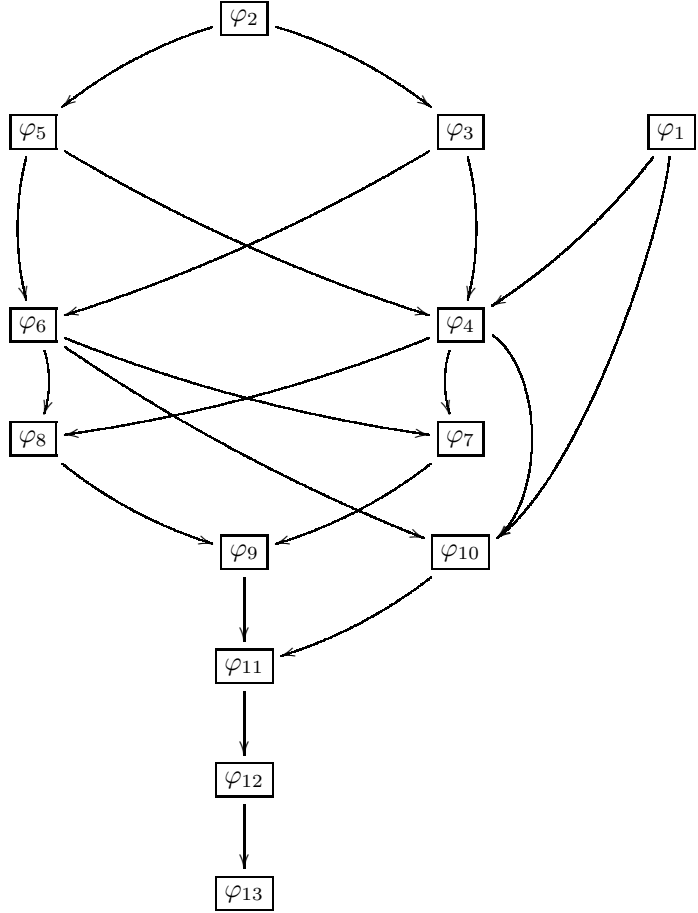
\begin{figure}[tb]
\caption{Contractions in $\mathcal{J}^4.$} \label{firstfig}
\medskip
$\xymatrix{
&&                                      && \ar@/_/[lld] \boxed{\varphi_2} \ar@/^/[rrd]\\
&& \boxed{\varphi_5} \ar@/_/[rrrrdd]\ar@/_/[dd] && && \boxed{\varphi_3}\ar@/^/[dd]\ar@/^/[lllldd]   && \boxed{\varphi_1} \ar@/^/[ddll]\ar@(d,r)[ddddll]\\\\
&& \boxed{\varphi_6} \ar@/_/[rrrrdd] \ar@/^/[d] \ar@/_/[rrrrd] &&      && \boxed{\varphi_4} \ar@/_/[d] \ar@/^/[lllld] \ar@(r,r)[dd]\\
&& \boxed{\varphi_8}     \ar@/_/[drr]            &&                          && \boxed{\varphi_7} \ar@/^/[dll]&&\\
&&                 &&   \boxed{\varphi_9}    \ar@//[d]  && \boxed{\varphi_{10}} \ar@/^/[dll]&& &&\\
&&                 &&   \boxed{\varphi_{11}} \ar@//[d]\\
&&                 &&   \boxed{\varphi_{12}} \ar@//[d]\\
&&                 &&   \boxed{\varphi_{13}}}$
\end{figure}

\bigskip

\bibliographystyle{amsplain}

\end{document}